\documentclass[a4paper]{amsart}

\newtheorem{thm}{Theorem}[section]
\newtheorem{cor}[thm]{Corollary}
\newtheorem{prop}[thm]{Proposition}
\newtheorem{lema}[thm]{Lemma}
\newtheorem{rem}[thm]{Remark}
\newcommand{\R}{\mathbb R}

\newcommand{\ve}{\varepsilon}

\newcommand{\lam}{\lambda}

\begin{document}

\title[Lyapunov Inequalities for PDE]
{Lyapunov-type Inequalities for Partial Differential Equations}

\author[P. L. de Napoli and J.P. Pinasco] {Pablo L. de Napoli and Juan
P. Pinasco}

\date{}
\thanks{ }

\begin{abstract}
In this work we present a Lyapunov inequality for linear and quasilinear elliptic
differential operators in $N-$dimensional domains $\Omega$. We also consider singular
and degenerate elliptic problems with $A_p$ coefficients involving the $p-$Laplace
operator with zero Dirichlet
 boundary condition.

  As an application of the inequalities obtained, we derive lower bounds for the first eigenvalue
 of the $p-$Laplacian, and compare them with the usual ones in the literature.
\end{abstract}

\maketitle

\section{Introduction}

In his classical work \cite{Ly}, Lyapunov proved that, given a continuous periodic and
positive function $w$ with period $L$, the solution $u$ of the ordinary differential
equation $u''+w(t)u = 0,$ in $(-\infty, +\infty)$, was stable if
$$ L \int_0^L w(t)dt< 4.$$

Then, Borg in \cite{Borg} introduced the Lyapunov inequality in his proof of the
stability criteria for sign changing weights $w$. He showed that the inequality
\begin{equation}\label{lyapclas}\frac{4}{L} \le \int_0^L |w(t)|dt
\end{equation}
  must be satisfied in order
to have a nontrivial solution in $[0,L]\subset \R$ of the problem
\begin{equation}\label{ecuac}\left\{ \begin{array}{l}
 u'' + w(t)u =0,\\
u(0)=0=u(L).\end{array}\right.
\end{equation}

Since then, it was rediscovered and generalized many times. Inequality
\eqref{lyapclas} was applied in stability problems, oscillation theory, a priori
estimates, other inequalities, and eigenvalue bounds for ordinary differential
equations. Different proofs of this inequality have been appeared in the literature:
the proof of Patula \cite{Pat} by direct integration, or the one of Nehari \cite{Neha}
showing the relationship with Green's functions, among several others. See the survey
\cite{BrHi} for other proofs.

In the nonlinear setting, the following inequality
\begin{equation}\label{lyap1d}
\frac{2^p}{L^{p-1}} \le \int_0^L w(t)dt
\end{equation}
generalized Lyapunov inequality \eqref{lyapclas} to $p-$Laplacian problems,
$$\left\{ \begin{array}{l}
 (|u'|^{p-2}u')'  + w(t)|u|^{p-2}u =0,\\
 u(0)=0=u(L).\\
\end{array}\right.
 $$
 Here,
$w\in L^1$ and $1<p<\infty$, for $p=2$ we recover the linear problem \eqref{ecuac}.
Several proofs were given in the last years, see \cite{LYHA, Pach, P, Y}; although it
seems to be derived first by Elbert \cite{Elb}.

Later, we extended it in \cite{DNP} to  nonlinear operators in Orlicz spaces
generalizing the $p$-Laplacian,
\begin{equation}\label{filap}-(\varphi(u'))' = \lam r(t)\varphi(u),
\end{equation}
where $\varphi(s)$ is a convex nondecreasing function, such that $s\varphi(s)$ satisfy
the $\Delta_2$ condition. Moreover, we also extend it to systems of resonant type (see
\cite{BoDF}) involving $p-$ and $q-$Laplacians in \cite{DNP2}.

\bigskip

Beside the one dimensional case, there are few works devoted to similar inequalities
for partial differential equations. An exception is the work of Ca\~nada, Montero and
Villegas \cite{C1, C2}, where the following problem was considered,
 \begin{equation}\label{1.1}
\left\{ \begin{array}{ll}
 \Delta u + w(x)u =0, &  x \in \Omega\\
\frac{\partial u}{\partial \eta}=0 , &  x \in \partial \Omega\end{array}\right.
\end{equation}
and a nonexistence result was obtained for general domains. The authors gives some
bounds involving the second Neumann eigenvalue $\mu_2$. However, it is well known that
$\mu_2$ fails to reflect geometric properties of $\Omega$, and can be made arbitrarily
close to zero by adding a slight perturbation of the domain as in \cite{CH}. Also,
several papers of Egorov and Kondriatev, included in their book \cite{Egor}, contain
Lyapunov type inequalities for higher order linear differential operators.

\bigskip

The aim of this work is to prove a Lyapunov inequality for $N$-dimensional (linear and
quasilinear) elliptic operators with zero Dirichlet boundary conditions, reflecting
more geometric information than the measure of the domain. Our toy model is the
$p$-Laplace operator, and we consider here the following problem,
 \begin{equation}\label{mainsyst}
\left\{ \begin{array}{ll}
 \Delta_p u + w(x)|u|^{p-2}u =0, &  x \in \Omega\\
u=0 , &  x \in \partial \Omega.\end{array}\right.
\end{equation}
As usual, we denote $\Delta_p u = div(| \nabla u|^{p-2}\nabla u)$ for any
$1<p<+\infty$, and the weight $w \in L^s$ for some $s$ depending on $p$ and $N$. We
include a short appendix with some facts about the eigenvalues of the $p$-Laplace
operator that we will need later.

Let us fix  the following notations that will be used below: let us call $r_{\Omega}$
the inner radius of $\Omega$,
$$ r_{\Omega}= \max_{x \in \Omega} d_\Omega(x) $$
where
$$ d_\Omega(x)= d(x,\Omega^c)= \inf_{y \in \partial \Omega} |x-y| $$
is the distance from $x\in \Omega$ to the boundary.

Now, let us note that the length $L$ of the interval in inequality \eqref{lyap1d} can
be thought as the \textit{measure} of the interval, but it can be understood also as
twice  the \textit{inner radius} of the interval, by rewriting the inequality as
$$2\left(\frac{2}{L}\right)^{p-1} \le \int_0^L q(t)dt.
$$
This is our main objective here: to derive some Lyapunov type inequalities involving
the inner radius of the domain and norms of the weight $w$.

\bigskip

We divide the paper in two main parts, in the first we cover the case $p>N$, and we
prove the existence of a Lyapunov inequality involving the $L^1$ norm of the weight
and the inner radius of the domain. We also consider singular problems, and we need to
prove a Morrey's theorem for $A_p$ weights.

In the second one we analyze the case $p< N$, we show that there are Lyapunov type
inequalities involving the $L^s$ norm for $ s > N/p$.

We do not consider here the case $p=N$. For $p=N=2$, we mention two interesting
results from Osserman \cite{Oss}:

\begin{thm}[Osserman, \cite{Oss}] Given a domain $\Omega \in \R^2$ of
  connectivity $k \ge 2$, the first Dirichlet eigenvalue of problem
$$ \left\{
\begin{array}{rclll}
-\Delta u  &=&    \lambda u & \hbox{in} &\Omega \\
u & = & 0 & \hbox{on} & \partial \Omega \\
\end{array}
\right. $$
satisfy
$$ \lambda_1 \ge \frac{1}{k^2 r_{\Omega}^2}.$$
\end{thm}

\begin{thm}[Osserman, \cite{Oss}]
Let $\Omega \in \R^2$, and $\Omega_{\ve}$ the domain obtained by removing from
$\Omega$ a finite number of disjoint disks of radius $\ve$  centered at a fixed set
$E$ of points in $\Omega$. Then,
$$ \lim_{\ve \to 0} \lambda_1(\Omega_{\ve}) = \lambda_1(\Omega).
$$
\end{thm}

Clearly, both results are enough to conclude that we cannot expect a general
inequality involving the inner radius of the domain when $p=N$, although it would be
very interesting to find a related inequality.

Finally, we show the optimality of the bounds, and we apply them to eigenvalue
problems. We compare them with Sturmian and isoperimetric bounds.

\section{Statement of the results and organization of the paper}

Let us state precisely our results in this Section.

In Section \S 3, we consider the case $p>N$ and we prove first:

\begin{thm}\label{Lyapunov-using-Morrey}
Let  $\Omega \subset \R^N$ be an open set, let $w \in L^1(\Omega)$ be a non-negative
weight, and let $u \in W^{1,p}_0(\Omega)$ with  $p>N$ be a nontrivial solution of
$$ \left\{
\begin{array}{rclll}
-\Delta_p u &=&   w(x) |u|^{p-2} u & \hbox{in} & \Omega \\
u & = & 0 & \hbox{on} & \partial \Omega \\
\end{array}
\right. $$
Then,
\begin{equation}\label{lyapmayor}  \frac{C}{r_{\Omega}^{p-N}} \leq  \| w\|_{L^1(\Omega)}
\end{equation}
where $C$ is an universal constant depending only on $p$ and $N$.
\end{thm}

Let us note that the constant $C$ is the same for any $\Omega \subset \R^N$, since it
is related to the constant given by Morrey's Theorem; we believe that it can be
improved for particular domains. However, the power of the inner radius is optimal.

\bigskip
Then,  we consider the following problem
$$
-div (v (x)|\nabla u|^{p-2}\nabla u) = w(x)|u|^{p-2}u
$$
where now $v$ is a singular or degenerate weight, typically a power of the distance to
the boundary or powers of $|x|$ (as in Henon equations, and Caffarelli-Kohn-Nirenberg
inequalities).

Here, the problem is more subtle since we need the density of continuous functions in
the weighted Sobolev space
$$
W_0^{1,p}(\R^N,v, w) := \{u\in L^1_{loc}(\R^N) : w^{1/p}u\in L^p(\R^N) \mbox{ and }
 v^{1/p}\nabla u \in [L^p(\R^N)]^N\}
$$
where $\nabla u$ is a distributional gradient in the sense of Schwartz.

Following \cite{Kil}, this is true when $v=w$ belong to the Muckenhoupt class $A_p$,
that is, $v$ is a nonnegative function in $L^1_{loc}(\R^N)$, and there exists a
constant $c_{p,v}$ such that
\begin{equation}
 \left( \int_B v(x)dx\right)\left( \int_B v(x)^{-\frac{1}{p-1}}dx\right)^{p-1} \le c_{p,v}|B|^p
\end{equation}
for every ball $B\in \R^N$.

The same argument applies for different weights $v$, $w$ in $A_p$, as we will show in
Lemma \ref{easylem} below. So, we will restrict ourselves to weights $v, w\in A_t$
with $t<p/N$, and in this case we prove the following Lyapunov type inequality:

\begin{thm}\label{singular}
Let $\Omega \subset \R^N$, and let $v \in A_t(\R^N)$, with $t<p/N$, and $v\ge 0$. Let
us define
$$
g(r_{\Omega})=\sup_{x\in \Omega} \int_{B(x,r_{\Omega})}v^{-\frac{1}{t-1}}(x) dx.
 $$

Let $u \in W^{1,p}_0(\Omega)$ be a nontrivial solution of
$$ \left\{
\begin{array}{rclll}
-div (v(x)|\nabla u|^{p-2}\nabla u) & = & w(x)|u|^{p-2}u & \hbox{in} & \Omega \\
u & = & 0 & \hbox{on} & \partial \Omega. \\
\end{array}
\right. $$
 Then, we have the following Lyapunov-type inequality
\begin{equation}\label{lyapsing}
1 \le C(p, t, N) \;    r_{\Omega}^{p-tN}  g(r_{\Omega})^{t-1}
\int_{\Omega} w(z) dz
\end{equation} where the constant
$C(p,t, N)$ depends only on $p$, $t$, and $N$.
\end{thm}

\bigskip

 Theorem \ref{singular} is based on the fact that $A_t \subset A_p$ whenever $t<p$. Briefly, we will bound
 $u$ by the fractional integral (or Riesz potential) of its gradient, and after adding the corresponding power of the coefficient,
  we wish
to use Holder's inequality with exponents $p$ in the gradient, and an exponent close
to $p'$ in $|\cdot |^{1-N}$.

\begin{rem} This theorem can be thought as a Morrey's embedding with $A_p$ weights. To our
knowledge, no such result was proved before for the case $p>N$. For $p<N$, we refer
the interested reader to the book of Turesson \cite{Ture}.
\end{rem}

Although the terms in the Lyapunov inequality \ref{lyapsing} seems difficult to
compute, in certain interesting case are rather simple to compute. We choose as an
example  a coefficient which is a power of the distance to the boundary, $v(x) =
d_{\Omega}^{\gamma}(x)$, and in this case we obtain a very clean bound,
$$
1 \le C  \;    r_{\Omega}^{p-N-\gamma} \int_{\Omega} w(z) dz,
$$
where $C$ depends only on $N$, $p$, and $\gamma$. Of course, $\gamma$ is restricted by
the $A_t$ condition, let us recall that $d^{\gamma}_{\Omega}(x)\in A_t$ for
$-1<\gamma< t-1$.

\bigskip
For $1<p < N$, a similar inequality cannot hold for arbitrary domains, as we mention
in the Introduction. Perhaps the easiest way to understand why is to remove a discrete
set of points with zero capacity from a ball, and the first eigenvalue
 remains the same.

So, in Section \S 4, we prove the following weaker inequality:

\begin{thm}\label{Lyapunov-pn}
Let $\Omega \subset \mathbb{R}^N$ be a smooth domain, $\frac{N}{p}<s$, and $w \in
L^s(\Omega)$.
Let $u \in W^{1,p}_0(\Omega)$ be a nontrivial solution of
$$ \left\{
\begin{array}{rclll}
-\Delta_p u &=&   w(x) |u|^{p-2} u & \hbox{in} & \Omega \\
u & = & 0 & \hbox{on} & \partial \Omega \\
\end{array}
\right. $$
Then, we have the following Lyapunov inequality
\begin{equation}\label{lyapmenor}  \frac{C}{r^{\frac{sp-N}{s}}} \leq  \| w
\|_{L^s(\Omega)}.
\end{equation}
The constant $C$ depends on $p$, $N$, and the capacity of $\mathbb{R}^N \setminus
\Omega$.
\end{thm}

 \bigskip
The proof of this theorem is based on the Sobolev immersion with critical exponent and
Hardy's inequality, and for this reason the $p-$capacity of $\mathbb{R}^N \setminus
\Omega$ appears on the constant. Although the constant is domain-dependent, for
certain classes of sets we can give an uniform constant, i.e., for Lipschitz or convex
domains, we have an explicit constant depending only on $p$ and $N$ (see the details
below at the end of Section \S4).

 \begin{rem} We do not consider singular problems when $p<N$. Similar results as in Section \S 3 can be
 obtained by combining the results in \cite{Ture} with Hardy-type inequalities
 involving
 $A_p$ weights, see
 the book of Opic and Kufner  \cite{OpK}, following the proof of Theorem
 \ref{Lyapunov-pn}.
 \end{rem}

\bigskip

Let us note that we have the following lower bounds for the first eigenvalue of the
$p-$Laplacian with zero Dirichlet boundary conditions:
\begin{cor}\label{core}
Let $\lambda_{1}$be the first eigenvalue of
\[
-\Delta_{p}u=\lambda w(x)|u|^{p-2}u,\] in $\Omega$ with zero Dirichlet boundary conditions in $\partial \Omega$.
 Then,
 \begin{itemize}
 \item for $p>N$ and $w$ as in Theorem \ref{Lyapunov-using-Morrey}, we have
$$
\frac{C}{r_{\Omega}^{p-N} \|w\|_1 }\leq
\lambda_{1},
$$
\item  for $p<N$ and $w$, $s$ as in Theorem \ref{Lyapunov-pn},
$$
\frac{C}{r_{\Omega}^{\frac{sp-N}{s}} \|w\|_s }\leq
\lambda_{1}.
$$
\end{itemize}\end{cor}
This Corollary follows directly from Theorems  \ref{Lyapunov-using-Morrey} and
\ref{Lyapunov-pn}, by replacing $w$ with $\lambda_1 w$.

\bigskip

In Section \S 5, we apply the bounds of Corollary \ref{core} to eigenvalue problems.

First, we show that the  powers of the inner radius appearing in Theorems
\ref{Lyapunov-using-Morrey} and  \ref{Lyapunov-pn}  are optimal:

\begin{prop}\label{optimal} Let $B(0,R)$ be the ball of radius $R$ centered at the
origin, and let
 $$\gamma = \left\{\begin{array}{lrl}
  p-N &  & \mbox{ if } p>N \\ \\
\displaystyle  \frac{sp-N}{s} & & \mbox{ if } p<N.\end{array}\right.$$

\begin{itemize}
\item Let $R>1$. For any $\beta < \gamma$, and $C$ fixed, there exists a
    non-negative weight $w$, and a solution $u_\beta \in W^{1,p}_0(B(0,R))$ of
$$ \left\{
\begin{array}{rclll}
-\Delta_p u  &=&   w(x) |u|^{p-2} u & \hbox{in} & B(0,R) \\
u & = & 0 & \hbox{on} & \partial B(0,R) \\
\end{array}
\right. $$
 such that the inequality
$$ \frac{C}{R^{\beta}} \leq  \| w \|_{L^1(B(0,R)} $$
 does not hold.

\item Let $R<1$. For any $\beta > \gamma$, and $C$ fixed,  there exists a
    non-negative weight $w$, and a solution $u_\beta \in W^{1,p}_0(B(0,R))$ of
$$ \left\{
\begin{array}{rclll}
-\Delta_p u  &=&   w(x) |u|^{p-2} u & \hbox{in} & B(0,R) \\
u & = & 0 & \hbox{on} & \partial B(0,R) \\
\end{array}
\right. $$
 such that the inequality
$$ \frac{C}{R^{\beta}} \leq  \| w \|_{L^1(B(0,R)} $$
 does not hold.
 \end{itemize}
\end{prop}

The result follows by computing a bound
 of the first eigenvalue of the $p$-Laplacian on a ball
  with a radial weight restricted to a small ball of radius
$\varepsilon$ for a suitable $\ve$.

 \bigskip

Finally,  we compare the lower bounds for the first eigenvalue of the $p-$Laplacian in
Corollary \ref{core} with the ones obtained with other techniques.

A classical tool for problems without weights is the Faber-Krahn inequality,
$$
\lambda_1(B) \le \lambda_1(\Omega),
$$
where $B$ is the ball with Lebesgue measure $|B|=|\Omega|$. Several proofs of this
inequality for the $p$-Laplacian appeared in the literature, and they are based on the
ideas of Talenti. Some improvements involving measures of the asymmetry of the domain
$\Omega$ are known, see \cite{Bat, tanos}.

For bounded weights, a Sturmian comparison argument combined with the variational
characterization of the first eigenvalue (see equation \eqref{varsist} in the
Appendix), enable us to replace $w$ with the norm $\|w\|_{L^{\infty}}$, obtaining now
lower bounds for $\lambda_1$.

For arbitrary weights, there are few inequalities involving their norms and the
measure of the domain, namely the works of Anane \cite{Ana} and Cuesta \cite{Cu}.

We show that for certain domains and weights, the bounds given by Lyapunov inequality
are better.

\bigskip
 We close the paper with a
short Appendix where we include some basic facts about $p-$Laplacian eigenvalues.

\section{Lyapunov's inequality for $p>N$}

Let us recall first Morrey inequality:
\begin{thm}[Morrey] If $p>n$, there exists a constant $C(N,p)$ such that for all
$u\in W_{0}^{1,p}(\Omega)$,
\[
|u(x)-u(y)|\leq C(n,p)\left\Vert \nabla u\right\Vert _{L^{p}}|x-y|^{\alpha}\,\] for
all $x$, $y\in\overline{\Omega}$, and $\alpha=1-\frac{N}{p}$.
\end{thm}

Now we are ready to prove Theorem \ref{Lyapunov-using-Morrey}.

\begin{proof}[Proof of Theorem \ref{Lyapunov-using-Morrey}]
  Let $u\in W_{0}^{1,p}(\Omega)$ be a
nontrivial solution of
$$-\Delta_{p}u=w(x)|u|^{p-2}u$$
with Dirichlet boundary conditions. Multiplying by $u$ and
integrating by parts, we obtain
 $$\int_{\Omega}|\nabla u|^{p}=\int_{\Omega}w(x)|u|^{p}.$$

 Since $p>N$, $u$ is continuous and let us choose
$c\in\Omega$ a the point of $\overline{\Omega}$ where $|u(x)|$ achieves its maximum.
Then, for $y=c$ and $x\in\partial\Omega$ we have that
\[
|u(c)|\leq C(N,p)\left(\int_{\Omega}|\nabla
u|^{p}dx\right)^{\frac{1}{p}}|x-c|^{\alpha}.\]

By using that
 $|x-c|\leq r_{\Omega}$, the inner radius of
$\Omega$, we get
\[
|u(c)|\leq C(N,p)\left(\int_{\Omega}w(x)|u|^{p}\,
dx\right)^{\frac{1}{p}}r_{\Omega}^{\alpha}.\]

 Hence,
\[
|u(c)|\leq C(N,p)|u(c)|\left(\int_{\Omega}w(x)dx\right)^{\frac{1}{p}}r_{\Omega}^{\alpha}\] and
cancelling out $|u(c)|$ we have the Lyapunov inequality
\[
\frac{1}{r_{\Omega}^{\alpha}}\leq C(N,p)\left(\int_{\Omega}w(x)dx\right)^{\frac{1}{p}},\]
 with $\alpha=1-\frac{N}{p}$.

\bigskip
 The
proof is finished.
\end{proof}

\begin{rem}\label{remarka}
In particular, let $\lambda_{1}$be the first eigenvalue of
\[
-\Delta_{p}u=\lambda w(x)|u|^{p-2}u \] in $\Omega$ with zero Dirichlet boundary
conditions in $\partial \Omega$. We have
\begin{equation}\label{cotaauto}
\frac{C(N,p)^{-p}}{r_{\Omega}^{p-N} \|w\|_1 }\leq
\lambda_{1} ,
\end{equation}
 which gives the lower bound for $\lambda_{1}$ in Corollary \ref{core}.
\end{rem}

\subsection{Singular and Degenerate Weights}

The following Lemma extend the results in \cite{Kil} for different weights in the
function and its distributional gradient:

\begin{lema}\label{easylem} For $v, w \in A_p$, the space
$W_0^{1,p}(\R^N, v, w)$ is the completion of $C_0^{\infty}(\R^N)$ with the norm
$$
\|\cdot \|_{p,v, w} : =  (\|\nabla \cdot\|^p_{[L^p(\R^N, v)]^N} + \|\cdot \|^p_{L^p(\R^N,w)})^{1/p}.
$$
\end{lema}

\begin{proof}
The proof follows by taking $u \in W_0^{1,p}(\R^N, v, w)$ and regularizing it by
convolution with a mollifier $\eta_j$. Now, from Lemma 1.5 in \cite{Kil},
\begin{align*}
\eta_j * u \to u & \qquad \mbox{ in } L^p(\R^N, w) \\
\nabla (\eta_j * u ) = \eta_j *\nabla u \to \nabla u & \qquad \mbox{ in } [L^p(\R^N,
v)]^N
\end{align*}
that is, $\eta_j *u\to u$ in $W_0^{1,p}(\R^N, v, w)$.
\end{proof}

We are ready to prove Theorem \ref{singular}.

\begin{proof}
Thanks to Lemma \ref{easylem}, we can choose a smooth function $u$. Now, given $x, y
\in \bar{\Omega}$, such that $r=|x-y| \le r_{\Omega}$, let us call $A=B(x,r) \cap
B(y,r)$. Hence,
\begin{align*}
|u(x) - u(y)| & \le \frac{1}{|A|}  \int_A |u(x)-u(z)|dz + \frac{1}{|A|}  \int_A |u(y)-u(z)|dz \\
\\ &\le  C   \int_{B(x,r)} \frac{ |\nabla u(z)| }{ |x-z|^{N-1} }dz +   C \int_{B(y,r)} \frac{
 |\nabla u(z)| }{ |y-z|^{N-1} }dz \\
\\  & = I_1+I_2
\end{align*}
where the constant $C$ depends only on $N$, see for instance, Evans \cite{Evans}.

Let us bound now $I_1$. We need to include the coefficient $v$ appearing in the
equation, and let us call $B=B(x,r)$. By using Holder's inequality:
 \begin{align*}
 I_1 &  =  C   \int_{B}
  \frac{|\nabla u(z)|}{|x-z|^{N-1}}v^{\frac1p}v^{-\frac1p}dz \\ \\
  &   \le  C  \left( \int_{B} v |\nabla u(z)|^p dz \right)^{\frac1p}
   \left(\int_{B} \frac{1}{|x-z|^{q(N-1)}} dz\right)^{\frac{1}{q}}
 \left(\int_{B} v^{-\frac{s}{p}} dz\right)^{\frac{1}{s}}
 \end{align*}
 where
 \begin{align*}
 & \frac1p+\frac1q+\frac1s=1,  \\ \\ & s = \frac{p}{t-1}.
 \end{align*}

Now, we have following bounds:
  \begin{align} &
\int_{B} v(z) |\nabla u(z)|^p dz   \le  \int_{\Omega} w(z)|u(z)|^p dz, \label{zerosob} \\  &
\int_{B} \frac{1}{|x-z|^{q(N-1)}} dz
      \le c r_{\Omega}^{q- qN+N}, \label{zerosob2}
 \\
 & \int_{B} v^{-\frac{s}{p}}(z) dz \le g(r_{\Omega}).  \label{zerosob3}
 \end{align}
 We have used that $v$ is positive, and by integrating by parts
the equation multiplied by $u$ in $\Omega$, we get  the first inequality. The second
one follows by integrating in polar coordinates in a bigger ball of radius
$r_{\Omega}$, the constant $c$ can be computed explicitly and depends only on $N$, $p$
and $q$. The last one was defined in this way in the hypotheses of the Theorem.

The bound for $I_2$ is almost identical, although we need first to impose some extra
condition on $u$. Since we are working in $W_0^{1,p}$, we can extend any function by
zero outside $\Omega$, and we can take a smooth function $u$ supported in $\Omega$.
So, we can integrate only over $B(y,r) \cap \Omega$ in the first inequality
\eqref{zerosob}, and we get
$$
|u(x) - u(y)| \le  C \;     r_{\Omega}^{1- N+\frac{N}{q}}  g(r_{\Omega})^{\frac{1}{s}}
\left(  \int_{\Omega} w(z)|u(z)|^p dz \right)^{\frac{1}{p}}
$$
where $C$ is a universal constant depending only on $N$, $p$ and $q$.

We are able to choose yet the points $x$ and $y$, and this is the last step of the
proof. Let $x$ be the point where $|u|$ is maximized, and $y$ one of the points in
$\partial \Omega$ which minimizes $|x-y|$. So, $u(y)=0$ and  $|x-y|< r_{\Omega}$.

After bounding $|u(z)| \le |u(x)|$ at the right hand side, and canceling out with the
one in the left hand side, we get
$$
1 \le C(p, t, N) \;    r_{\Omega}^{p- pN+\frac{pN}{q}}  g(r_{\Omega})^{\frac{p}{s}}
  \int_{\Omega} w(z) dz .
$$

Finally, let us observe that the  relationship between Holder's exponent implies that
 $$ \frac{p}{q} = p-t,  \qquad \frac{p}{s} =t-1.
$$
The proof is finished.
\end{proof}

\begin{rem}
Let us note that inequality \eqref{zerosob2} holds when $q- qN+N>0$, and $q\ge p'$ in
Holder's inequality. That is,
$$
\frac{p}{p-1} < q <\frac{N}{N-1}
$$
which makes sense  because $p>N>1$.

On the other hand, the bigger is $q$, the bigger is $s$. When $q \to \frac{N}{N-1}$,
we have that $s\to \frac{pN}{p-N}$, and the integral in inequality \eqref{zerosob3} is
well defined when $v\in A_t$ with $$t< p/N.$$
\end{rem}

As an application of Theorem \ref{singular} we have the following result for
quasilinear problems involving the distance to the boundary.

\begin{prop}
Let $\Omega\in \R^N$ a bounded open set, $p>N$, and $u\in W_0^{1,p}(\Omega,
d^{\gamma}, w)$ a nontrivial solution of
 $$
-div(d_{\Omega}^{\gamma}(x)|\nabla u|^{p-2}\nabla u) = w(x)|u|^{p-2}u
$$
in $\Omega$ with zero Dirichlet boundary conditions in $\partial\Omega$, where
$d_{\Omega}(x)$ is the distance to the boundary. Then,
$$
1 \le C  \;    r_{\Omega}^{p-N-\gamma} \int_{\Omega} w(z) dz,
$$
where $C$ depends only on $N$, $p$, and $\gamma$.
\end{prop}

In order to prove this Proposition, we can repeat the previous proof, although only
inequality \eqref{zerosob3} depends on $d^{\gamma}_{\Omega}$. So, we will  improve
this bound by integrating in  $B(x, d_{\Omega}(x))$ instead of $B(x, r_{\Omega})$.

\begin{proof}  We divide the proof in two cases, depending on the sign of $\gamma$.

First, we consider $\gamma <0$. Given  $z\in \Omega$, we choose $y\in \partial \Omega$
with $r=|x-y| = d_{\Omega}(x)$, clearly we have $r \le r_{\Omega}$. After a
translation if necessary, we can suppose that $y=0$, and we have $d_{\Omega}(z) \le
|z|$, and then
$$ d_{\Omega}^{-\frac{s\gamma}{p}}(z) \ge |z|^{-\frac{s\gamma}{p}}.$$
Hence, we can estimate $g(r_{\Omega})$ by computing
$$
\int_{B(x, r)}
d_{\Omega}^{-\frac{s\gamma}{p}}(z)   \le \int_{B(x,r) }  |z|^{-\frac{s\gamma}{p}}dz =
 r^{N-\frac{s\gamma}{p} } \int_{B(x/r,1)} |\eta|^{-\frac{s\gamma}{p}}d\eta  \le  C r_{\Omega}^{N-\frac{s\gamma}{p} },
 $$
where in the last step we changed variables, $\eta=z/r$.

So, we can bound
 $$\int_{B(x, r)}
d_{\Omega}^{-\frac{s\gamma}{p}}(z)   \le C \; r_{\Omega}^{N-\frac{s\gamma}{p} }.
 $$

\bigskip
Let us consider now $\gamma>0$. Given  $z\in \Omega$  and $y\in \partial \Omega$ with
$r=|x-y|= d_{\Omega}(x) \le r_{\Omega}$ as before, clearly we have $r \le r_{\Omega}$.
After a translation if necessary, we can suppose that $x=0$, and we have
$d_{\Omega}(z) \ge d_{\partial B(0,r)}(z)$, the distance to the boundary of the ball.

Then, since  $\gamma >0$,
$$d_{\Omega}^{-\frac{s\gamma}{p}}(z) \le d_{\partial
B(0,r)}^{-\frac{s\gamma}{p}}(z),$$ and
 \begin{align*}
 \int_{B(0, r)} d_{\Omega}^{-\frac{s\gamma}{p}}(z)  & \le \int_{B(0,r) }  (r-|z|)^{-\frac{s\gamma}{p}}dz
  \\
 & =c_N \int_0^r (r-\rho)^{-\frac{s\gamma}{p}} \rho^{N-1} d\rho \\
& = c_N r^{N-\frac{s\gamma}{p}} \int_0^1 (1-\hat{\rho})^{-\frac{s\gamma}{p}} \hat{\rho}^{N-1} d\hat{\rho} \\
& = C \; r^{N-\frac{s\gamma}{p}}.
 \end{align*}
Again, we have the bound
 $$
\int_{B(0, r)} d_{\Omega}^{-\frac{s\gamma}{p}}(z)  \le  C \; r_{\Omega}^{N-\frac{s\gamma}{p} }.
 $$

The last step is to replace this bound instead of the power of $ g(r_{\Omega})$ in
Lyapunov's inequality given by Theorem \ref{Lyapunov-using-Morrey}. By using that $p/s
=t-1$, we have
 $$
1 \le C  \;    r_{\Omega}^{p-N-\gamma} \int_{\Omega} w(z) dz
 $$
 and the
proof is finished.
\end{proof}

\section{Lyapunov-type inequality for $p<N$}

Let us prove now Theorem \ref{Lyapunov-pn}.

\bigskip
\begin{proof} Let us define
$$q= \alpha p + (1-\alpha)p^*,$$
where $p^*$ is the Sobolev conjugate exponent, and  $\alpha \in (0,1)$ which will be
chosen later.

Then, we have
$$ \frac{1}{r_{\Omega}^{\alpha p}}
 \int_\Omega |u|^q \; dx   \leq  \int_\Omega
\frac{|u|^q}{d(x)^{\alpha p}} \; dx  ,
$$
where $d(x)$ is the distance from $x$ to the boundary. Now, Holder's inequality with
exponents $1/\alpha$ and $(1/\alpha)' = 1/(1-\alpha)$ gives
\begin{equation}\label{tohardysob}
  \int_\Omega
\frac{|u|^{\alpha p}|u|^{(1-\alpha)p^*}}{d(x)^{\alpha p}} \; dx
\leq \left( \int_\Omega
\frac{|u|^{ p}}{d(x)^{p}} \; dx \right)^{\alpha}\left( \int_\Omega
|u|^{p^*}  \; dx \right)^{1-\alpha}.
\end{equation}

Let us recall Hardy and Sobolev inequalities,
$$
\int_\Omega \frac{|u|^{ p}}{d(x)^{p}} \; dx \le C_h \int_\Omega |\nabla u|^p \; dx,
$$
$$
\int_\Omega
|u|^{p^*}  \; dx  \le C_s \left(\int_\Omega |\nabla u|^p \; dx\right)^{p^*/p}
$$
and by using them in equation \eqref{tohardysob}, we get
$$
\left( \int_\Omega
\frac{|u|^{ p}}{d(x)^{p}} \; dx \right)^{\alpha}\left( \int_\Omega
|u|^{p^*}  \; dx \right)^{1-\alpha} \le C_{hs} \left( \int_\Omega |\nabla u|^p \; dx
\right)^{\alpha + (1-\alpha)p^*/p}
$$
where $C_{hs}$ is a constant depending only on $C_h$ and $C_s$, the constants involved
in Hardy and Sobolev inequalities.

Hence, by using the weak formulation for equation $ -\Delta_p u = w(x) |u|^{p-2}u$,
and applying again Holder's inequality with exponents $s$ and $s'$ we obtain
 \begin{align*}  \left( \int_\Omega |\nabla u|^p \; dx \right)^{\frac{\alpha p + (1-\alpha)p^*}{p}}
 & =
 \left( \int_\Omega w(x)|u|^p \; dx \right)^{\frac{\alpha p + (1-\alpha)p^*}{p}}  \\
& \leq   \left( \int_\Omega w(x)^{s} \right)^{\frac{\alpha p + (1-\alpha)p^*}{ps}}
  \left( \int_\Omega |u|^{ps'}
\; dx \right)^{\frac{\alpha p+ (1-\alpha)p^*}{ps'}}.
\end{align*}

We choose now $\alpha$ such that $ps'=q$. Let us observe that
$$
 \frac{\alpha p+ (1-\alpha)p^*}{ps'} = 1,
 $$
  $$
  \frac{\alpha p + (1-\alpha)p^*}{ps}=  \frac{s'}{s},
  $$
and
$$
\alpha = \frac{p^*-ps'}{p^*-p}.
$$

Finally, we get
$$ \frac{1}{r_{\Omega}^{\alpha p}}
 \int_\Omega |u|^q \; dx   \leq
\|w\|_{L^s}^{s'}
   \int_\Omega |u|^q \; dx,
$$
and the Theorem is proved.
\end{proof}

\begin{rem}
A tedious computation shows that
 $$\frac{\alpha p}{s'}  = \frac{p}{s'}\frac{p^*-ps'}{p^*-p} =  \frac{sp-N}{s}.
$$

Since $s>N/p$, the exponent is positive.
\end{rem}

\bigskip

\begin{rem} The constant $C$ depends on the constant $C_h$ appearing on the Hardy
inequality. When $\Omega$ is convex, we have  $C_h = \left(\frac{p}{N-p}\right)^p$;
for other domains, the constant depends on the capacity of  $\mathbb{R}^N \setminus
\Omega$; for Lipschitz domains the constant is close to $1/2$, see \cite{haz, Ker} for
details.
\end{rem}

\begin{rem}\label{remarkb}
In particular, let $\lambda_{1}$be the first eigenvalue of
\[
-\Delta_{p}u=\lambda w(x)|u|^{p-2}u \] in $\Omega$ with zero Dirichlet boundary
conditions in $\partial \Omega$. We have
\begin{equation}\label{cotaautob}
\frac{C}{r_{\Omega}^{\frac{sp-N}{s}} \|w\|_s }\leq
\lambda_{1} ,
\end{equation}
 which gives the lower bound for $\lambda_{1}$ in Corollary \ref{core}.
\end{rem}

\section{Some applications to eigenvalue problems}

\subsection{Optimality of the bounds.}
Let us show the optimality of the power of the inner radius appearing in the
inequality.

\begin{proof}[Proof of Proposition \ref{optimal}.]

For brevity, we will consider only the case $p>N$, $R>1$ since the remaining ones
follow exactly in the same way.

Fix $R>1$, and  let us show that the bound \eqref{cotaauto} from Remark \ref{remarka}
cannot hold for some power $\beta < p-N$ and
$$
 w(r)=\chi_{[0, \ve]}(r) r^{1-N},
$$
where $\chi_{[0, \ve]}(r)$ is the characteristic function of $[0, \ve]$.

Clearly, $\|w\|_1 = \omega_{N-1}\ve$, where $\omega_{N-1}$ is the surface measure of
the unit ball, since
$$
 \int_{B(0,R)} \chi_{[0,\ve]}(|x|)|x|^{1-N}dx=\int_{\omega_{N-1}}\int_0^{\ve}r^{1-N}r^{N-1}dr d\theta.
$$

Let $\lam_1^{(R)}$ and $\lam_1^{(\ve)}$ be the first eigenvalues of the $p$-Laplacian
problem
$$-\Delta_{p}u=\lam w(x)|u|^{p-2}u$$
with Dirichlet boundary conditions
 in $B(0,R)$ and $B(0,\ve)$ respectively.
 We have $\lam_1^{(R)} < \lam_1^{(\ve)}$, since extending the
 functions by zero, we have $W_0^{1,p}(B(0,\ve))\subset W_0^{1,p}(B(0,R)),$ and the
 inequality follows by using the variational characterization,
\begin{align*}
& \lambda_1^{(R)}  = \inf_{\{u\in W_0^{1,p}(B(0,R)) : u\not\equiv 0\}}\frac{\int_{B(0,R)} |\nabla u|^p
dx}{\int_{B(0,R)}\chi_{[0,\ve]}(|x|)|x|^{1-N}dx }
\\
\\  & \lam_1^{(\ve)}   = \inf_{\{u\in W_0^{1,p}(B(0,\ve)): u\not\equiv 0\}}  \frac{\int_{B(0,\ve)}
 |\nabla u|^p dx}{\int_{B(0,\ve)} |x|^{1-N}dx }.
\end{align*}

Since the first eigenfunction in a ball is radial,
\begin{align*} \lambda_1^{(R)}   \le \lam_1^{(\ve)}
& = \inf_{\{u\in  W^{1,p}(0,\ve) : u(\ve)=0, u\not\equiv 0\} \}}
 \frac{\int_0^\ve r^{N-1}|u'|^p dr}{\int_0^{\ve}|u|^p dr} \\
\\
& \le \ve^{N-1} \frac{\pi^p_p}{\ve^p}.
\end{align*}

Then,
$$
 \frac{C}{R^{\beta}}\le \lambda_1 \omega_{N-1}\ve.
$$

Let $\ve=R^{\alpha}$, and if we can choose $\alpha<1$ such that $\beta-
\alpha(p-N)<0$, we reach a contradiction:
$$R^{\alpha(p-N)} \le c R^{\beta}$$

However, this is equivalent to find $\alpha$ satisfying
$$
0 < \frac{\beta}{p-N}< \alpha< 1,
$$
and we can find it if
$$
\frac{\beta}{p-N}<1,
$$
which holds exactly when  $\beta<p-N$.
\end{proof}

\begin{rem}
Clearly, $\beta>\gamma$ is of no interest when the inner radius is greater than 1,
since we get a worse bound instead of an improvement. Similar observations hold for
the remaining cases.
\end{rem}

\subsection{Comparison with other estimates}

Let us consider the following eigenvalue problem:
 \begin{equation}\label{1.2}
\left\{ \begin{array}{ll}
 -\Delta_p u = \lambda w(x)|u|^{p-2}u, &  x \in \Omega\\
  u =0 , &  x \in \partial \Omega\end{array}\right.
\end{equation}

There are few ways to obtain lower bounds for the eigenvalues of the $p-$Laplacian. In
the constant coefficient case, we can use symmetrization and then compare with the
first eigenvalue of a ball with the same measure as $\Omega$, since the Faber-Krahn
inequality implies
$$\lam_1(B) \le \lam_1(\Omega).$$

For weighted problems, a Sturmian-type comparison theorem is available, that is, if
$w_1(x) \le w_2(x)$, then
$$
 \lambda_k(w_2) \le \lambda_k(w_1),
$$
since the eigenvalues are computed with the Rayleigh quotient. Also, Anane and Cuesta
obtained some inequalities that we will review below.

In the rest of the section we compare those bounds with the one obtained from
Corollary \ref{core} when $p>N$ and, $N=2$. Similar results hold for $p<N$, and higher
dimensions.

\bigskip
 {\bf Faber-Krahn.} In order to compare Faber-Krahn inequality and Lyapunov inequality  \eqref{lyapmayor},
 we can expect that the former will be worse in thin domains. So,
 let us take the following family of domains in $\R^2$
 $$
\Omega_R=\{(x,y)\in \R^2 \, : \, 0\le x\le R, \; 0\le y \le 1/R\}
$$
with $0<R\le 1$.

Since $|\Omega_R|=1$, Faber-Krahn gives a fixed lower bound for any $\Omega_R$.
However, Lyapunov inequality (with $w\equiv 1$) implies
 $$
 \frac{C(2,p)^{-p}}{r_{\Omega_R}^{p-2} \|w(x)\|_1 }\leq
  \frac{C(2,p)^{-p}}{(R/2)^{p-2} } = \frac{C}{R^{p-2}} \le \lam_1.
  $$
Now, from equations \eqref{equivnorm}, when $R\to 0$,
 $$ \hat\lambda_1 =
\frac{\pi_p^p}{R^p}+ \pi_p^pR^p  = O\left(\frac{\pi_p^p}{R^p}\right),$$
 and by using  \eqref{otroautopseudo} from Appendix,
 $$
 \lambda_1 =
  O\left(\frac{\pi_p^p}{R^p}\right).
  $$

Lyapunov inequality is better for $R$ small, although it is not optimal in this family
of sets.

Faber-Kahn inequality can be improved as in \cite{Bat, tanos}.  Following Fusco, Maggi
and Pratelli,
$$
 \lam_1(\Omega) \ge \lam_1(B) \left\{ 1 + \frac{A(\Omega)^{2+p}}{C(N,p)}\right\},
$$
where $C(N,p)$ is a fixed constant, and $A(E)$ is the Fraenkel asymmetry of a set $E$
with finite measure,
$$
 A(E) := \inf\left\{ \frac{|E\Delta(x_0+rB(0,1))|}{|E|} : x_0 \in \R^N, r^N|B(0,1)|=|E|\right\}.
$$
Since $A$ is bounded above by 2, the maximum constant that can be involved in the
lower bound is independent of $R$ for the previous family of sets.

\bigskip
 {\bf Sturm type bounds.} Intuitively, this kind of bounds can be improved because by
 adding a highly concentrated spike with very low mass in a given weight we can
 change slightly the eigenvalue, and the supremum norm of the weight can be made
 arbitrarily big. The proof follows easily by using the eigenfunction of the
 unperturbed weight as a test function.

 However, the improvement can be better, even for domains with an inner radius of  the same order than
 the diameter of the domain.
 Suppose that $0\le w\le M$, $\Omega=[0,R]\times[0,R]$, and $R\gg 1$, with $\int_{\Omega} w(x)=1$.
 The variational characterization of the first eigenvalue,
together with \eqref{equivnorm} and  \eqref{otroautopseudo} implies
$$
\frac{ 2\pi_p^p}{M R^p} \le
\lambda_1.
$$
Now, Lyapunov inequality gives the bound
 $$
 \frac{C}{R^{p-2}} \le \lam_1.
 $$
 Let us observe that the difference between them not depend only on $M$, but on a
 factor $M R^2$. Indeed, we always have
$$R^{p-2}\int_{\Omega} wdx \le R^{p}M. $$

 \bigskip
 {\bf Bounds involving norms of the weights.}
For arbitrary weights, there are few estimates involving their norms and the measure
of the domain.

First, Anane obtained  in \cite{Ana} the following estimate:
$$
\frac{C}{|\Omega|^{\sigma}||w||_\infty} \le  \lambda ,
$$
where
\begin{align*} \sigma =p/N & \qquad \mbox{if} \quad 1<p\le N, \\ \sigma=1/2 & \qquad \mbox{if} \quad
N<p.
\end{align*}
Also, Cuesta proved in \cite{Cu} the following inequality:
$$\frac{C}{|\Omega|^{\frac{sp-N}{sN}}||w||_s} \le  \lambda ,
$$
where
\begin{align*} s>N/p & \qquad \mbox{if} \quad 1<p\le N, \\ s=1 & \qquad \mbox{if} \quad
N<p.
\end{align*}
 Clearly, they are Lyapunov type inequalities, involving the measure of the
domain instead of the inner radius. Those inequalities were widely used to show that
the first eigenvalue is isolated, since any other eigenfunction has at least two nodal
domains, and one of them must shrink, but the inequality implies that the first
eigenvalue of the shrinking domain cannot converge to the first eigenvalue of the full
domain.

Let us observe that
$$
 |\Omega|^{1/N} \ge C r_{\Omega}
$$
with equality only when $\Omega$ is a ball, so Corollary \ref{core} gives better
bounds, except in Anane's bound for $p>N$, which is better when $w\simeq cte$,
$|\Omega| \simeq r_{\Omega}^N$, and the measure of $\Omega$ is small enough.

\appendix
\numberwithin{equation}{section}

\setcounter{equation}{0}

\section{Eigenvalues of the $p$-Laplacian}

 We say that a function $u$ is an
eigenfunction of problem
 \begin{equation}\label{eigenpro}
\left\{ \begin{array}{ll}
 -\Delta_p u = \lambda w(x)|u|^{p-2}u, &  x \in \Omega\\
  u =0 , &  x \in \partial \Omega\end{array}\right.
\end{equation}
corresponding to the eigenvalue $\lambda$ if
$$
\int_{\Omega} |\nabla u|^{p-2} \nabla u \nabla \varphi \; dx =
\lambda \int_{\Omega} w(x) |u|^{p-2} u  \varphi \, dx
$$
for any test-function $\varphi\in W^{1,p}_0(\Omega)$. The existence of infinitely many
eigenvalues was proved by Garcia Azorero and Peral Alonso in \cite{GAPA} by using the
critical point theory of Ljusternik--Schnirelmann, and the variational
characterization given by the Rayleigh quotient,
\begin{equation}\label{varsist}
\lam_k=\inf _{C\in \mathcal{C}_k} \sup _{u\in C}
\frac{  \int_{\Omega}|\nabla u|^{p}\, dx }{
 \int_{\Omega} w(x) |u|^p dx},
\end{equation}
where $\mathcal{C}_k$ is the class of compact symmetric ($C=-C$) subsets of
$W^{1,p}_0(\Omega)$ of Krasnoselskii genus greater or equal that $k$, see \cite{Rabi}
for details.

It is well known that the first eigenfunction is positive and simple, see for instance
\cite{Ana}. Indeed, this result holds for more general operators, including the
so-called pseudo $p$-Laplacian operator,
$$
-\hat\Delta_p v := -\sum_{i=1}^{N} \frac{\partial}{\partial x_i} \left(
\left| \frac{\partial v }{\partial x_i}\right|^{p-2} \frac{\partial v}{\partial
x_i}\right),
$$
and the proof is exactly the same, the simplicity follows by a Picone type identity,
and the positivity by considering $|u_1|$ as a test function, where $u_1$ is the first
eigenfunction.

We will use the pseudo $p$-Laplacian in order to control the eigenvalues of the
$p$-Laplacian. The equivalence of norms in $\R^N$, $ |x|_q \le C_{p,q} |x|_p$ enable
us to compare the first eigenvalue of each problem, since both can be defined
$$
\hat\lambda_1 =\inf_{u \in B}\| |\nabla u|_p \|_p^p; \qquad \lambda_1
=\inf_{u\in B}\| |\nabla u|_2 \|_p^p.
$$
where
$$B=\{u\in W_0^{1,p}(\Omega) : \int_{\Omega} w(x) |u|^p \; dx\}
$$

Clearly,
 \begin{equation}\label{equivnorm}\begin{array}{rcll} \hat\lambda_1 & \le
\lambda_1\le &
N^{(p-2)/2}  \hat\lambda_1 & \mbox{ if } 2<p,\\
 N^{(p-2)/2} \hat\lambda_1 & \le
\lambda_1 \le &
 \hat\lambda_1 &  \mbox{ if } p<2.\end{array}
 \end{equation}

\bigskip
The first eigenvalue of the one dimensional problem with  $w \equiv 1$
 \begin{equation}\label{eigenpro2}
\left\{ \begin{array}{l}
-(|u'|^{p-2}u')' = \lam |u|^{p-2}u \qquad in  (0,L) \\
u(0)=u(L)=0\end{array}\right.
\end{equation}
 can be computed explicitly with the
help of the function $\sin_p(x)$, defined implicitly as
$$
x =  \int_0^{\sin_p(x)} \Big(\frac{p-1}{1-t^p}\Big)^{1/p} dt.
$$
 and its first zero $\pi_p$
$$
\pi_p = 2 \int_0^1 \Big(\frac{p-1}{1-t^p}\Big)^{1/p} dt.
$$
We have
$$
\lambda_1 = \frac{\pi_p^p}{L^P}.
$$
Also, for the mixed boundary condition $u'(0)=u(L)=0$, the first eigenvalue is given
by
$$
\lambda_1 = \frac{2^p \pi_p^p}{L^P}.
$$
We refer the interested reader to the work of  Del Pino, Drabek and Manasevich,
\cite{DDM} for more details about the one dimensional case.

\bigskip
Finally, for $w \equiv 1$ the first eigenvalue $\hat\lambda_1$ and the corresponding
eigenfunction $\hat{u_1}$ of the pseudo $p$-Laplacian in a cube $Q=[0,L]^N\subset
\R^N$ can be computed explicitly. Following \cite{FBP}, we have
 \begin{equation}\label{autopseudo}
\hat\lambda_1 =\frac{\pi_p^p N}{L^p} , \qquad \hat u_1(x)=\prod_{j=1}^N\sin_p\left(\frac{\pi_p x_j}{L}\right),
 \end{equation}
which combined with inequalities \eqref{equivnorm} gives upper and lower bounds for
the first eigenvalue of the $p$-Laplacian in $Q$ with $w(x)\equiv 1$.

A similar computation gives, for $\Omega = \prod_{j=1}^N [0,L_i]$,
 \begin{equation}\label{otroautopseudo}
\hat\lambda_1 =\sum_{j=1}^N \frac{\pi_p^p}{L_j^p}  , \qquad \hat u_1(x)=\prod_{j=1}^N\sin_p\left(\frac{\pi_p x_j}{L_j}\right)
 \end{equation}

\section*{Acknowledgements}

This work was partially supported by Universidad de Buenos Aires under grant
20020100100400, by CONICET (Argentina) PIP 5478/1438, and ANPCyT PICT 07 910. The
authors are members of CONICET.

\appendix
\numberwithin{equation}{section}

\setcounter{equation}{0}

\bigskip
\address{Pablo Luis de Napoli   \hfill\break
Departamento  de Matem\'atica, \hfill\break IMAS - CONICET \hfill\break FCEyN  UBA
\hfill\break Ciudad
Universitaria
 \hfill\break  Av. Cantilo s/n  (1428)  \hfill\break
Buenos Aires, Argentina. \hfill\break e-mail: {\tt pdenapo@dm.uba.ar}}

\bigskip
\address{Juan Pablo Pinasco \hfill\break
Departamento  de Matem\'atica, \hfill\break IMAS - CONICET \hfill\break FCEyN  UBA
\hfill\break Ciudad Universitaria
 \hfill\break  Av. Cantilo s/n  (1428)   \hfill\break
Buenos Aires, Argentina.
  \hfill\break e-mail: {\tt
jpinasco@dm.uba.ar}}

\end{document}